\documentclass[reqno,12pt,letterpaper]{amsart}

\usepackage{amsmath,amssymb,amsthm,mathtools,microtype}
\usepackage[usenames,dvipsnames]{xcolor}

\usepackage[colorlinks=true,linkcolor=Red,citecolor=Green]{hyperref}

\usepackage{graphicx}
\usepackage{tikz}
\usetikzlibrary{arrows.meta}

\allowdisplaybreaks

\newtheorem{theorem}{Theorem}
\newtheorem{proposition}{Proposition}
\newtheorem{lemma}[proposition]{Lemma}

\numberwithin{equation}{section}

\newcommand{\R}{\mathbb R}
\newcommand{\C}{\mathbb C}
\renewcommand{\Im}{\operatorname{Im}}
\renewcommand{\Re}{\operatorname{Re}}
\newcommand{\im}{\Im}
\newcommand{\re}{\Re}
\newcommand{\supp}{\operatorname{supp}}
\newcommand{\norm}[1]{\lVert #1\rVert}
\newcommand{\ip}[2]{\langle #1,#2\rangle}

\title[Bouncing-ball modes]
{Bouncing-ball modes in ergodic billiards}

\author{Zhenhao Li}
\email{zli1010@berkeley.edu }

\author{Maciej Zworski} 
\email{zworski@berkeley.edu}
\address{Department of Mathematics, University of California, Berkeley, CA 94720}

\begin{document}
\begin{abstract}
We show that for a horizontally stretched Bunimovich billiard, {\em for almost every stretching parameter}, there exists a sequence of eigenfunctions
concentrating in the rectangular part. 
Thanks to the work of Markarian–Oliffson Kamphorst–Pinto de Carvalho, we know that for an open set of parameters such billiards are ergodic.

The proof is a result of interaction with ChatGPT 6: we were curious if it could establish the existence of bouncing ball modes for Bunimovich or Sinai billiards. That was not successful until we suggested varying the wings. That produced an essentially complete argument for the existence of quasimodes. We then suggested using Hassell's parameter-variation argument which led to the result about eigenfunctions. The argument was then significantly simplified and clarified by the authors.

\end{abstract}

\maketitle

\section{Introduction}
\label{s:int}

The celebrated theorem of Shnirelman \cite{Sh74} states that classical
ergodicity forces most high energy eigenfunctions to become uniformly
distributed. For billiards, this means that an orthonormal basis of
Dirichlet eigenfunctions contains a density-one subsequence whose
probability densities converge to normalized area measure. That
case was proved by G\'erard--Leichtnam \cite{GL} for domains including Bunimovich billiards and by
Zelditch--Zworski \cite{ZZ96} for manifolds with arbitrary piecewise smooth boundaries.

The theorem allows however the possibility of exceptional sequences
concentrating on sets of zero phase space volume. That was established by Hassell \cite{Ha} who showed that for almost every length of the rectangle ($t$ in 
Figure~\ref{fig:billiard} with $ a = 1$) such a sequence exists for the Bunimovich stadium, that is, it is not quantum unique ergodic (QUE). 
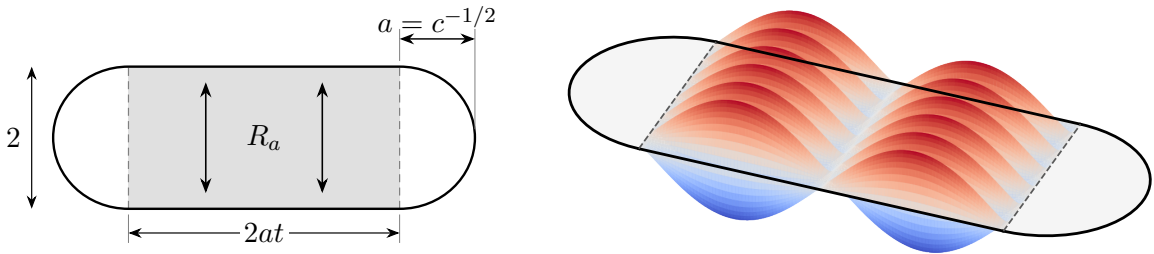
\begin{figure}[!b]
\centering
\begin{minipage}[c]{0.47\linewidth}
\centering
\begin{tikzpicture}[x=0.94cm,y=0.94cm,>=Stealth,font=\small]
 \def\billa{1.06}
 \def\billL{1.908} 
 \path[use as bounding box] (-3.53,-1.62) rectangle (3.42,1.78);
 \fill[black!12] (-\billL,-1) rectangle (\billL,1);
 \draw[densely dashed,black!50,line width=.45pt]
       (-\billL,-1)--(-\billL,1) (\billL,-1)--(\billL,1);
 \draw[line width=.9pt]
       (-\billL,1)--(\billL,1)
       arc[start angle=90,end angle=-90,x radius=\billa,y radius=1]
       --(-\billL,-1)
       arc[start angle=-90,end angle=-270,x radius=\billa,y radius=1];
 \draw[<->,line width=.6pt] (-.82,-.78)--(-.82,.78);
 \draw[<->,line width=.6pt] (.82,-.78)--(.82,.78);
 \node at (0,0) {$R_a$};
 \draw[<->,line width=.45pt] (-3.27,-1)--node[left] {$2$}(-3.27,1);
 \draw[black!50,line width=.35pt] (-\billL,-1.08)--(-\billL,-1.48)
                    (\billL,-1.08)--(\billL,-1.48);
 \draw[<->,line width=.45pt] (-\billL,-1.34)--node[fill=white,inner sep=1.6pt] {$2at$}(\billL,-1.34);
 \draw[black!50,line width=.35pt] (\billL,1.08)--(\billL,1.53)
                      ({\billL+\billa},.12)--({\billL+\billa},1.53);
 \draw[<->,line width=.55pt] (\billL,1.39)--({\billL+\billa},1.39);
 \node[above,inner sep=1.5pt] at ({\billL+.5*\billa},1.42) {$a=c^{-1/2}$};
\end{tikzpicture}
\end{minipage}\hfill
\begin{minipage}[c]{0.51\linewidth}
\centering
\includegraphics[width=\linewidth]{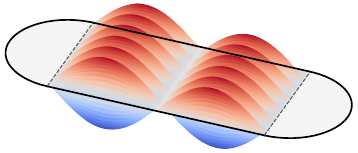}
\end{minipage}
\caption{The elliptical stadium $S_a$, $a=c^{-1/2}$ and two bouncing-ball trajectories. Right:  $q_6$ from \eqref{eq:seed}, zero in the wings; Theorem \ref{t:main} shows that it produces accurate quasimodes.}
\label{fig:billiard}
\end{figure}

That sequence was not specified in \cite{Ha} and a 
natural candidate comes from the bouncing-ball trajectories perpendicular to
the parallel sides of its rectangle. We prove the existence of such
sequences for almost every member of the family with elliptical wings
shown in Figure~\ref{fig:billiard}.

We state a rough version of the main result as follows:
\begin{theorem}\label{t:1}
Fix $ t> 0 $ and consider the billiard $ S_a $ shown in  Figure~\ref{fig:billiard}: 
\[ S_a := \left\{ ( x, y ) : |y | < 1, |x| < at + a \sqrt{1 -y^2 } \right\}, \ \ \ R_a := S_a \cap \{ |x| <at \} 
. \]
There are integers $ m_k \to \infty $ such that for almost every $ a > 0 $, there exist
normalised Dirichlet eigenfunctions
\begin{equation}\label{eq:efphys}
 - \Delta \psi_k=\lambda_k\psi_k,\qquad
 \lambda_k=(m_k\pi)^2+ {\pi^2}/{a^2t^2} +o(1), \qquad\| \psi_k \|_{ L^2 ( S_a ) } = 1, 
 \end{equation}
such that for every fixed $\varepsilon>0$,
\begin{equation}
\label{eq:wings}
 \int_{S_a\cap\{|x|\geq at+\varepsilon\}}|\psi_k|^2d xd y
 \to0. 
\end{equation}
In other words, the mass of $ \psi_k $ goes to $ 0 $ outside any neighbourhood of the rectangle $ R_a $. 
\end{theorem}
A more precise version justifying the bouncing ball title is presented in \S \ref{s:efloc}. The remark after Theorem \ref{t:main} gives a statement valid for 
a family of fixed internal rectangles. 

The limit \eqref{eq:wings} in Theorem \ref{t:1} should be compared to a result of Burq--Zworski \cite{buzy05} which states that for {\em any} sequence of normalised eigenfunctions and any $ \varepsilon > 0 $ there exists $ C_\varepsilon > 0 $ such that 
\begin{equation}
\label{eq:rect} \liminf_{ k \to \infty } \int_{ S_a 
\cap {\{ |x| \geq at - \varepsilon \}} } | \psi_k |^2 dx dy > \frac{1}{C_\varepsilon}. 
\end{equation}

Markarian--Oliffson Kamphorst--Pinto de Carvalho \cite[Theorem 2]{MKC} proved ergodicity of the billiard flow for $ S_a $ under the condition that
\begin{equation}\label{eq:erg}
 1<a<\sqrt{4-2\sqrt2},\qquad at>2a^2\sqrt{a^2-1}.
\end{equation}
The range of $ a$'s and $ t$'s is shown in Figure~\ref{f:erg}. 

\begin{figure}
\includegraphics[width=14cm]{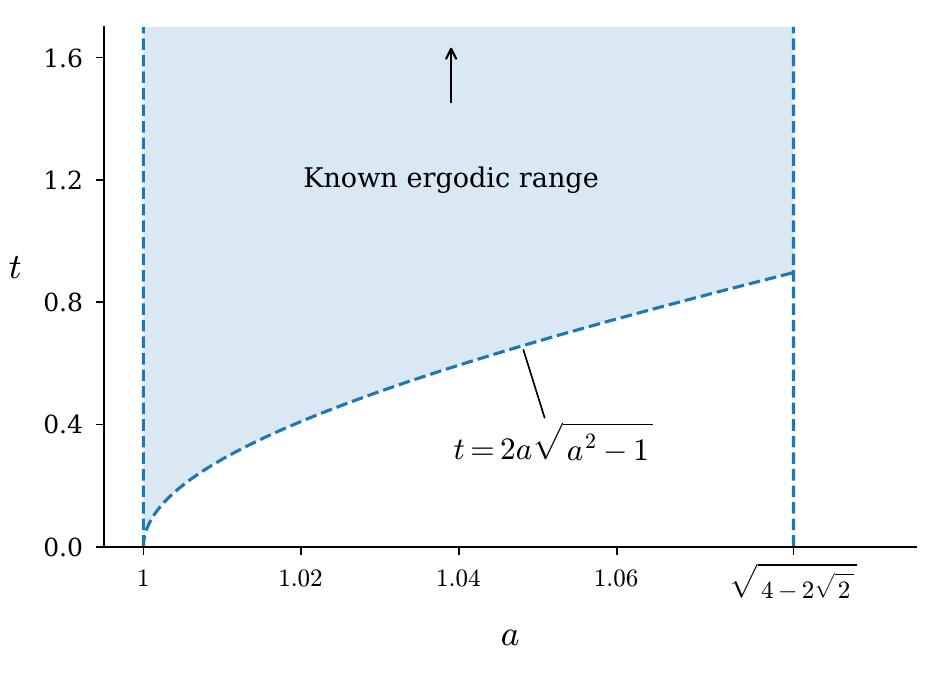}
\caption{\label{f:erg} An illustration of the condition 
\eqref{eq:erg} from \cite{MKC} on $ (a , t )$ guaranteeing ergodicity of the billiard flow for $ S_a $ in Figure~\ref{fig:billiard}.}
\end{figure}

The theorem is a consequence of estimates needed to establish existence of good quasimodes obtained from bouncing ball modes. For that it is convenient to work in a fixed domain and consider a quarter billiard, which we will do from now on. Thus we fix $ t > 0 $ and define  $\Omega $ to be the interior of $\overline R\cup\overline C$, where 
\begin{equation}\label{eq:geom}
 \begin{gathered}
 R=(-t,0)\times(0,1),\ \  \Gamma=\{0\}\times(0,1),\ \
 C=\{x>0,\ y>0,\ x^2+y^2<1\}.
 \end{gathered}
\end{equation}
For real $c>0$, let $H_c$ be the self-adjoint Dirichlet realization of
\begin{equation}\label{eq:H}
 H_c=-c\partial_x^2-\partial_y^2.
\end{equation}
This is unitarily equivalent to the Euclidean Dirichlet Laplacian on the horizontally stretched quarter-stadium obtained from $ S_a $ by choosing
$ a = 1/\sqrt c $. 

The bouncing ball modes in $ R $ extended by $ 0 $ to all of $ \Omega $ is given by 
\begin{equation}\label{eq:seed}
 \begin{gathered}
q_m(x,y):=\sqrt{\frac2t}\sin\frac{\pi(x+t)}t\,e_m(y)\,\boldsymbol 1_R, \ \ \ 
 e_m(y): =\sqrt2\sin(m\pi y), \ \ \ \norm{q_m}=1 . 
 \end{gathered}
\end{equation}

\begin{theorem}\label{t:main}
There are integers $m_k\to\infty$ such that for almost every real $c>0$ there exist $u_k\in H^2 ( \Omega ) 
\cap H^1_0(\Omega) $ satisfying 
\begin{equation}\label{eq:qm}
 \norm{u_k-q_{m_k}}\to0,\qquad
 \norm{(H_c-(m_k\pi)^2-c\nu)u_k}\to0, \qquad \nu := \pi^2/t^2 . 
\end{equation}
\end{theorem}
We stress that in Theorem \ref{t:1} (and Theorem \ref{t:ef}) we do not assert
$\norm{\varphi_k-q_{m_k}}\to0$ for individual eigenfunctions.

\noindent
{\bf Remark.}
One can also keep the physical rectangle fixed. Put $L=at$ and write
\[
S_{L,a}=\{(x,y):|y|<1,\quad |x|<L+a\sqrt{1-y^2}\}.
\]
For almost every $L>0$, the existence and concentration conclusions of
Theorems~\ref{t:1} and \ref{t:main} (and of Theorem \ref{t:ef}) hold on $S_{L,a}$ for almost every
$a>0$. Thus the rectangle $(-L,L)\times(-1,1)$ is fixed and only the
elliptical wings vary. The spectral centres are now
$ 
 E_m=(m\pi)^2+{\pi^2}/{L^2}$, independent of $a$, and the eigenvalues in Theorems~\ref{t:1} and \ref{t:ef} satisfy
$\lambda_k=E_{m_k}+o(1)$. The subsequences may depend on $a$.
Indeed, for every fixed $t>0$ the exceptional set of $c>0$ has measure
zero. After identifying the domains, the norm and spectral conditions
in the proofs are measurable in $(c,t)$, so Fubini's theorem gives a null
exceptional set of pairs.

\subsection{Outline of the proof.} 
The functions $q_m$ solve the eigenvalue equation except for a jump of
normal derivatives at $\Gamma$. To remove it, we solve
\[
 (H_c-(m\pi)^2-z)U_m=e_m\delta_\Gamma,
 \qquad g_m(c,z): =\ip{U_m|_\Gamma}{e_m}_{L^2(\Gamma)}.
\]
For real $c>0$ and $z=c\nu+i$, the energy identity gives
$\Im\,g_m=\norm{U_m}^2$. The correction
$v_m=q_m+c\gamma_tU_m(c,c\nu+i)$ with $\gamma_t := \sqrt{2}\, \pi t^{-3/2}$ then satisfies
$(H_c-(m\pi)^2-c\nu)v_m=i(v_m-q_m)$. Thus the smallness of $ g_m $ 
is the key to obtaining a quasimode from $ q_m $. 
The bounded errors sufficient for \cite{Ha} would not give the
concentration obtained here -- making them tend to zero is the main
additional difficulty.

To achieving this smallness we complexify  $c$. At $c=1-i$ and $z=i$, the energy identity
also controls the unscaled horizontal derivative $\partial_xU_m$.
With this additional control, a semiclassical propagation argument,
in the spirit of Burq--Zworski \cite{buzy05,buzy04} used to obtain
\eqref{eq:rect}, forces the mass in the cap to vanish. A trace estimate
then gives $g_m(1-i,i)\to0$.
Positivity and ``harmonic measure magic'' transfer this decay to
\[
 \int_I \Im\,g_m(c,c\nu+i)\,dc\to 0,
 \qquad I\Subset(0,\infty).
\]
Related uses of harmonic measure in spectral theory go back to
Pearson \cite{Pe}. This averaged estimate gives the required quasimodes
after selecting a subsequence, for almost every $c$.

Finally, Hassell's parameter-variation argument \cite{Ha}, applied in
shrinking spectral windows, selects exact eigenfunctions satisfying
$\lambda_k^{-1}\norm{\partial_x\varphi_k}^2\to0$. Propagation results for semiclassical defect measures from Burq \cite{Bu97} then show the concentration on the closed bouncing-ball
family.

We remark that the argument applies to more general partially rectangular domains as in Hillairet--Hassell \cite[Appendix]{Ha}.
To keep the presentation focused we restrict ourselves to the case of stretched Bunimovich billiards.

\subsection{Review of the literature}
Here we will review some physics and maths papers, which in addition to the already cited \cite{Ha},\cite{buzy05}, addressed the bouncing ball problem. We do not claim completeness, especially when it comes to the review of the physics literature.

\subsubsection{Physics}
Early work of Shapiro--Taylor--Brumer \cite{STB84} used an adiabatic approximation to describe regular low energy eigenstates of the stadium. Bai et al.\ \cite{BHST85} developed a Born--Oppenheimer construction of bouncing ball modes, which were also observed in the numerical study of McDonald--Kaufman \cite{MK88}. O'Connor--Heller \cite{OH88} gave semiclassical arguments for the persistence of bouncing ball localization at arbitrarily high energies. These modes are associated with the neutral family of trajectories between the parallel sides. This differs from the scarring along unstable periodic orbits introduced by Heller \cite{He84}. Sieber et al.\ \cite{SSCL93} derived semiclassically the oscillatory contribution of bouncing ball trajectories to the spectral counting function and its effect on spectral statistics, accounting for features of a measured stadium spectrum.

B\"acker--Schubert--Stifter \cite{BSS97} and Tanner \cite{Ta97} predicted that the number of bouncing ball modes below energy $E$ in the stadium grows as $E^{3/4}$, using semiclassical constructions supported by numerical calculations. The effect of these exceptional modes on the rate of quantum ergodicity was studied in \cite{BSS98}. B\"acker--Schubert \cite{BS99} studied numerically the localization of bouncing ball modes in position and momentum space. L\"ock--B\"acker--Ketzmerick \cite{LBK12} analysed their coupling to chaotic states and obtained counting laws from the resulting decay rates. More recently, Dietz--Nguyen--\v{C}ade\v{z} \cite{DNC26} numerically compared bouncing ball modes in Schr\"odinger, neutrino and graphene stadium billiards.

On the experimental side, Stein--St\"ockmann \cite{SS92} visualized bouncing ball modes in a microwave stadium cavity (see their Figure~2a). Chinnery--Humphrey \cite{CH96} visualized bouncing ball modes in an acoustic stadium, while Gr\"af et al.\ \cite{Graf92} and Alt et al.\ \cite{Alt99} studied the spectral contributions of bouncing ball trajectories in microwave stadium billiards. The work of Heller--Crommie--Lutz--Eigler \cite{HCLE94} on standing electron waves in quantum corrals provides a related physical realization of the stadium geometry. Ge et al.\ \cite{Ge24} reported bouncing ball scars in stadium-shaped graphene quantum dots with electrostatic confinement --  see in particular the numerically computed states in their Figure~4f, compared with the experimental pattern in Figure~2b.

\subsubsection{Mathematics}
Zelditch \cite{Ze04} provided a conceptual precursor of Hassell's  argument by showing that 
the failure of QUE follows from the existence of bouncing ball modes within an approximate span of a uniformly bounded number of actual modes. In \cite{buzy05}, \eqref{eq:rect} was established, showing that modes cannot concentrate in subsets of the rectangle. That was refined by Burq--Hassell--Wunsch \cite{BHW07} who gave lower bounds (going to zero with energy) on the leakage into the wings. Papers by Hillairet--Marzuola \cite{HM12} and Christianson \cite{Ch15} provided additional quantitative restrictions on bouncing-ball quasimodes. 

The breakthrough by Hassell \cite{Ha} showed the failure of QUE for almost every Bunimovich stadium. (Comments on that and earlier results can be found on Tao's blog \cite{Tao07,Tao08}.) Mangoubi--Weller Weiser \cite{MW24} clarify some boundary regularity issues in Hassell's argument at the junction of the straight and curved bounary pieces. Most recently, 
Li--Liu \cite{LL26} used Hassell's argument and bouncing ball modes to show physical-space non-equidistribution for almost every parameter, without however showing concentration in the rectangle.

\noindent {\sc Acknowledgements.}
In addition to the interaction described in the abstract, we also thank ChatGPT 6 for producing the 
figures and help in reviewing the literature. A partial support from the Simons
Foundation through Targeted Grant Award No.~896630, ``Moir\'e Materials
Magic'' is also gratefully acknowledged. 

\section{An auxiliary boundary value problem}\label{s:aux}
If we consider $q_m$ in
\eqref{eq:seed} as an element of $ H_0^1 ( \Omega ) $, the mismatch of normal 
derivatives at $ \Gamma $ gives
\begin{equation}\label{eq:defect}
 (H_c-(m\pi)^2-c\nu)q_m=-c\gamma_t e_m\delta_\Gamma, \ \ \ 
 \gamma_t: = \sqrt{2}\, \pi t^{-3/2},
\end{equation}
where $e_m\delta_\Gamma$ is the boundary functional
$ V\mapsto\ip{e_m}{V|_\Gamma}_{L^2(\Gamma)}$.
To remove this term, we first solve the problem with the source given by 
$ e_m \delta_\Gamma $, a complex spectral parameter $ z $, and complexified $ c$:
Thus, for
\begin{equation}\label{eq:par}
 \Re c>0,\qquad \Im c\leq0,\qquad
 z\in\C_+:=\{z\in\C:\Im z>0\},
\end{equation}
we let $U=U_m(c,z)\in H^1_0(\Omega)$ satisfy
\begin{equation}\label{eq:source}
 (H_c-(m\pi)^2-z)U=e_m\delta_\Gamma\quad\text{in }\Omega.
\end{equation}
The restrictions $U_{m,C}=U_m|_C$ and $U_{m,R}=U_m|_R$ solve the
homogeneous equation in their interiors, vanish on the exterior
boundary, and satisfy transmission conditions
\begin{equation}\label{eq:jump}
 \begin{gathered}
 U_{m,R}|_\Gamma=U_{m,C}|_\Gamma,\qquad
 c\partial_xU_{m,R}|_\Gamma-c\partial_xU_{m,C}|_\Gamma=e_m.
 \end{gathered}
\end{equation}
The conormal derivatives are understood weakly. We define the following crucial object 
which is the $ e_m $ component of the restriction to $ \Gamma$:
\begin{equation}\label{eq:g}
 g_m(c,z)=\ip{U_m|_\Gamma}{e_m}_{L^2(\Gamma)}.
\end{equation}

To show that this transmission problem has a solution, we rephrase it
weakly using a quadratic form:
\begin{equation}\label{eq:weak}
 \begin{split}
 c\ip{\partial_xU}{\partial_xV}+\ip{\partial_yU}{\partial_yV}
       -((m\pi)^2+z)\ip UV
   =\ip{e_m}{V|_\Gamma}_{L^2(\Gamma)},
 \end{split}
\end{equation}
for all $V\in H^1_0(\Omega)$. Inner products are linear in the first
factor and the norms and inner products without a specified domain are over
$\Omega$.

\begin{lemma}\label{l:pos}
For the parameters in \eqref{eq:par}, problem \eqref{eq:weak} has a
unique solution. The functions $U_m(c,z)$ and $g_m(c,z)$ are holomorphic on
\[
 \{(c,z)\in\C^2:\Re c>0,\ \Im c<0,\ \Im z>0\},
\]
and extend holomorphically to a neighborhood of each $(c_0,z_0)$ with
$c_0>0$ real and $z_0\in\C_+$. Moreover, for $ g_m $ defined in \eqref{eq:g}, 
\begin{equation}\label{eq:pos}
 \Im g_m(c,z)=(-\Im c)\norm{\partial_xU_m(c,z)}^2
                  +(\Im z)\norm{U_m(c,z)}^2\geq0.
\end{equation}
In particular, for fixed real $c>0$, the function
$z\mapsto g_m(c,z)$ is holomorphic on $\C_+$ and has nonnegative
imaginary part there.
\end{lemma}

\noindent
{\bf Remarks.} 1. We stress that no uniformity in $ m $ is asserted in the conclusions of the lemma. It is an abstract functional analysis argument.

\noindent 2. The energy identity \eqref{eq:pos} explains the significance $ g_m $ defined in \eqref{eq:g}. For $ c > 0 $
(the physically interesting case of billiards) and $z=c\nu+i$, the identity $\Im\,g_m(c,z)=\norm{U_m(c,z)}^2$ shows that $\Im\,g_m\ll1$ implies $\|U_m(c,z)\|\ll1$. That means that the correction to \eqref{eq:defect}
$ r_m := c \gamma_t U_m ( c, z ) $ would be small and and
$ q_m + r_m $ would produce a quasimode.

\begin{proof}[Proof of Lemma \ref{l:pos}]
The right hand side of \eqref{eq:weak},
$V\mapsto\ip{e_m}{V|_\Gamma}_{L^2(\Gamma)}$, is a continuous antilinear
functional on $H^1_0(\Omega)$.

If the left hand side of \eqref{eq:weak} is denoted by $B(U,V)$, then for a sufficiently
large positive $M$, depending on $m,c,z$,
\[
 \begin{split}
 \Re\big((1+iM)B(U,U)\big)
  &=(\Re c-M\Im c)\norm{\partial_xU}^2+\norm{\partial_yU}^2\\
  &\ \ \ \ \  +(M\Im z-(m\pi)^2-\Re z)\norm U^2
 \end{split}
\]
is coercive on $H^1_0(\Omega)$.
Consequently the Lax--Milgram theorem (see for instance
\cite[\S 6.2.1, Theorem 1]{Evans}) gives existence and
uniqueness.
The form $B$ defines a bounded operator
$\mathcal B(c,z):H^1_0(\Omega)\to H^{-1}(\Omega)$ by
$\bigl(\mathcal B(c,z)U\bigr)(V)=B(U,V)$, where $H^{-1}(\Omega)$
is the continuous antilinear dual of $H^1_0(\Omega)$.
It depends holomorphically, in operator norm, on $c,z$ and so does the
inverse guaranteed by the Lax--Milgram argument. Writing
\[
 U_m(c,z)=\mathcal B(c,z)^{-1}(e_m\delta_\Gamma),
\]
we see that $(c,z)\mapsto U_m(c,z)$ is holomorphic with values in
$H^1_0(\Omega)$.
Applying the continuous linear functional
$U\mapsto\ip{U|_\Gamma}{e_m}_{L^2(\Gamma)}$ gives the same conclusion
for $g_m$. The invertibility holds for $c,z$ satisfying
\eqref{eq:par} and hence the holomorphy holds in a neighbourhood of
any such point.

Taking $V=U$ in \eqref{eq:weak} gives $B(U,U)=\bar g_m$, from which
\eqref{eq:pos} follows.
\end{proof}

\section{The complex-parameter estimate}\label{s:cx}
We now prove that the boundary coefficient \eqref{eq:g} tends to zero
at one fixed pair of complex parameters:
\begin{equation}\label{eq:dec}
 g_m(1-i,i)\to0,\qquad m\to\infty.
\end{equation}
This decay and ``harmonic measure magic" will be used to obtain an average decay of $\Im\,g_m(c,c\nu+i)$ for the relevant positive values of $ c$.
The particular values $1-i$ and $i$ in \eqref{eq:dec} are inessential:
any fixed pair with $\Re c>0$, $\Im c<0$, $\Im z>0$ works, with
constants depending on that pair.


\subsection{Uniform bounds without a high-frequency loss}
\label{s:unifb}
For $u\in H^1(C)$ vanishing on the circular arc, $ x = \sqrt{ 1 - y^2 }$, integration in $ x $ over $ 0 < x< \sqrt{1- y^2 } $
gives
\begin{equation}\label{eq:tr}
 \begin{split}
 \norm{u|_\Gamma}_{L^2(\Gamma)}^2
  &=-2\Re\int_C \partial_xu\,\overline ud xd y
  \leq2\norm u_{L^2(C)}\norm{\partial_xu}_{L^2(C)}.
 \end{split}
\end{equation}
In this section we write $U_m=U_m(1-i,i)$ and $g_m=g_m(1-i,i)$. Using the definition \eqref{eq:g}, formula \eqref{eq:pos}, and the estimate \eqref{eq:tr} gives
\begin{equation}\label{eq:bnd}
 \begin{split}
 |g_m|^2
 &\leq\norm{U_{m,C}|_\Gamma}_{L^2(\Gamma)}^2
\leq\norm{U_{m,C}}_{L^2(C)}^2+
          \norm{\partial_xU_{m,C}}_{L^2(C)}^2\\
 &\leq\norm{U_m}_{L^2(\Omega)}^2+
          \norm{\partial_xU_m}_{L^2(\Omega)}^2
   =\Im g_m\leq |g_m|,
 \end{split}
\end{equation}
We conclude that
\begin{equation}\label{eq:unif}
 |g_m|\leq1,\qquad
 \norm{U_m}_{L^2(\Omega)}^2+
 \norm{\partial_xU_m}_{L^2(\Omega)}^2\leq1.
\end{equation}
Taking the real part of \eqref{eq:weak}, with $V=U_m$, also gives
\begin{equation}\label{eq:yb}
 \begin{split}
 &(m\pi)^{-2}\big(\norm{\partial_yU_{m,C}}_{L^2(C)}^2+
                    \norm{\partial_yU_{m,R}}_{L^2(R)}^2\big)\\
 &\quad=\norm{U_m}_{L^2(\Omega)}^2+(m\pi)^{-2}\Re g_m
        -(m\pi)^{-2}\norm{\partial_xU_m}_{L^2(\Omega)}^2
        \leq 2, 
 \end{split}
\end{equation}
where we used \eqref{eq:unif} to get the estimate. 
We note that the estimates on $ \partial_x U_m $, unlike those on $ \partial_y U_m $, do not involve any loss when $ m \to \infty $. That is what we mean by the title of this subsection.

\subsection{Cap compactness: a semiclassical argument}
\leavevmode
We take the semiclassical approach of \cite[\S\S 6.1,6.2] {buzy04} and study the measures obtained
in the limit at $ h := 1/m\pi \to 0 $. The key component is an adaptation 
of Burq's transverse reflection argument
\cite[Appendix]{Bu97}: the horizontal derivative bound permits
only vertical directions, which reflection at the arc cannot preserve.
\begin{lemma}\label{l:cap}
Suppose $u_m\in H^1(C)$ vanishes on the arc and on $y=0$, satisfies
\begin{equation}\label{eq:cap}
 (-(1-i)\partial_x^2-\partial_y^2-(m\pi)^2-i)u_m=0\quad\text{in }C,
\end{equation}
and, for a constant $K$ independent of $m$,
\[
 \norm{u_m}_{L^2(C)}+\norm{\partial_xu_m}_{L^2(C)}
       +\norm{(m\pi)^{-1}\partial_yu_m}_{L^2(C)}\leq K,
 \qquad \norm{u_m|_\Gamma}_{L^2(\Gamma)}\leq K.
\]
Then 
\begin{equation}
\label{eq:l2}  \norm{u_m}_{L^2(C)}\to  0 . \end{equation}
\end{lemma}
\begin{proof}
Odd reflection of $u_m$ across $y=0$ extends it to $H^1(D)$, where
\begin{equation}\label{eq:D}
 D=\{(x,y):x>0,\ x^2+y^2<1\}.
\end{equation}
The extension solves \eqref{eq:cap} in $D$ and vanishes on the circular
arc.  Putting $ h = 1/m \pi $ and $ u = u_m $, we can rewrite the lemma hypotheses as
\begin{equation}\label{eq:sc}
 \begin{gathered}
 (P_{0,h}-iQ_h)u=0,\qquad
 P_{0,h}=-h^2\Delta-1,\qquad Q_h=-h^2\partial_x^2+h^2,\\
 \norm u+\norm{u_x}+\norm{hu_y}+\norm{u|_{x=0}}\leq C.
 \end{gathered}
\end{equation}
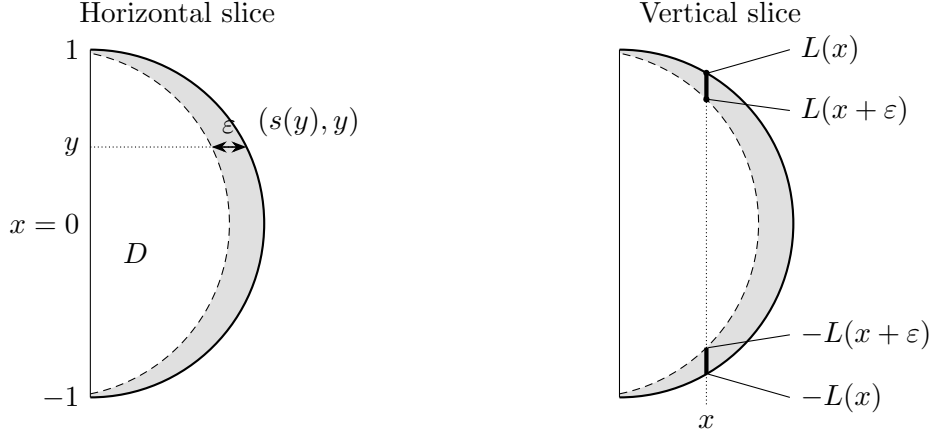
\begin{figure}[tbp]
\centering
\begin{tikzpicture}[x=2.5cm,y=2.5cm,>=Stealth,font=\small]
 \begin{scope}
  \path[fill=black!12] (0,-1)
    plot[domain=-90:90,samples=100] ({cos(\x)},{sin(\x)})
    -- (0,1)
    plot[domain=90:-90,samples=100] ({max(0,cos(\x)-0.20)},{sin(\x)}) -- cycle;
  \draw[thick] (0,-1) arc[start angle=-90,end angle=90,radius=1];
  \draw[thin] (0,-1)--(0,1);
  \draw[densely dashed] plot[domain=-78.463:78.463,samples=90]
        ({cos(\x)-0.20},{sin(\x)});
  \node at (.26,-.17) {$D$};
  \node[left] at (0,0) {$x=0$};
  \node[left] at (0,1) {$1$};
  \node[left] at (0,-1) {$-1$};
  \draw[densely dotted] (0,.44)--(.897998,.44);
  \draw[<->,semithick] (.697998,.44)--(.897998,.44);
  \node[above] at (.797998,.46) {$\varepsilon$};
  \node[left] at (0,.44) {$y$};
  \node[above right] at (.897998,.44) {$(s(y),y)$};
  \node at (.50,1.22) {Horizontal slice};
 \end{scope}
\end{tikzpicture}
\caption{The collar $\mathcal A_\varepsilon$ (shaded) in the
half-disk $D$.}
\label{fig:collar}
\end{figure}
We first exclude mass in a shrinking boundary layer. Set
$s(y)=\sqrt{1-y^2}$ and denote the horizontal collar of the arc by
$ \mathcal A_\varepsilon
   =\{(x,y)\in D:0<s(y)-x<\varepsilon\}$ (see Figure \ref{fig:collar}) .
For fixed $y$, apply the fundamental theorem of calculus to $u$, using $u(s(y),y)=0$, and then Cauchy--Schwarz:
\[
 |u(x,y)|^2
   =\left|\int_x^{s(y)}\partial_x u (r,y)d r\right|^2
   \leq(s(y)-x)\int_x^{s(y)}|\partial_x u(r,y)|^2d r.
\]
Integrating $x$ over  $(\max(0,s(y)-\varepsilon),s(y))$, and in $ y $ over $(-1,1)$, gives
\begin{equation}\label{eq:arc}
 \int_{\mathcal A_\varepsilon}|u|^2d xd y
 \leq\tfrac12{\varepsilon^2}
       \norm{\partial_x u}_{L^2(D)}^2\leq C\varepsilon^2.
\end{equation}
Similarly, $u(x,y)=u(0,y)+\int_0^x\partial_xu (r,y)d r$
gives
\begin{equation}\label{eq:strip}
 \begin{split}
 \int_{\{0<x<\varepsilon\}\cap D}|u|^2d xd y
 &\leq2\varepsilon\norm{u|_{x=0}}_{L^2(-1,1)}^2
       +\varepsilon^2\norm{\partial_x u}_{L^2(D)}^2
       \leq C\varepsilon.
 \end{split}
\end{equation}
Together these bounds rule out the concentration of mass on $\partial D$ in the semiclassical limit, meaning they exclude loss of $L^2$ mass to $\partial D$.
The localized energy identity
\[
 \int_D\chi|h\nabla u|^2\,dxdy
 =\int_D\chi|u|^2\,dxdy+O_\chi(h^2),
 \qquad\chi\in C_c^\infty(\{x>0\};\R),
\]
also excludes concentration of $h\nabla u$ on the arc.
Thus any subsequence defining an interior semiclassical measure (see \cite[\S 5.1]{Zw} for an introduction and \cite[Appendix]{Bu97} for boundary case) satisfies
\begin{equation}\label{eq:capmu}
 \mu(T^*D)=\lim_{h\to0}\norm u^2,\qquad
 \supp\mu\subset\{\xi=0,\ \eta=\pm1\},
\end{equation}
where the support restriction follows from $\norm{hu_x}=O(h)$ and
the equation.

To adapt Burq's propagation argument, we choose a self-adjoint order-zero test
operator $A_h$ supported away from $x=0$ and tangential near the arc,
so that $A_hu$ has zero trace there. The commutator calculus gives
\[
 \norm{\partial_x(A_hu)}
 \leq C\bigl(\norm{u_x}+\norm u_{H_h^1}\bigr)=O(1).
\]
Hence
\begin{equation}\label{eq:abs}
 \frac1h\ip{Q_hu}{A_hu}
 =h\bigl(\ip{u_x}{\partial_x(A_hu)}+\ip u{A_hu}\bigr)=O(h).
\end{equation}
The complex term therefore disappears from the limiting commutator
identities, giving the usual interior propagation and equality of
incoming and outgoing fluxes at transverse reflection
\cite[Appendix, second step]{Bu97}.

Every vertical ray through $D$ meets the arc transversally at
$n_xn_y\ne0$. Its reflected horizontal component is
$-2\eta n_xn_y\ne0$. But that contradicts  \eqref{eq:capmu}. Thus $\mu=0$
and $\norm u\to0$, as claimed in  \eqref{eq:l2}. 
\end{proof}

Apply Lemma~\ref{l:cap} to $U_{m,C}(1-i,i)$. Its equation in the cap
follows from \eqref{eq:weak}, and its bounds from
\eqref{eq:bnd}--\eqref{eq:yb}. Hence
\[
 \norm{U_{m,C}(1-i,i)}_{L^2(C)}\to0.
\]
The trace inequality \eqref{eq:tr} now gives
\[
 |g_m(1-i,i)|^2
 \leq2\norm{U_{m,C}(1-i,i)}_{L^2(C)}
        \norm{\partial_xU_{m,C}(1-i,i)}_{L^2(C)}\to0.
\]
Finally, \eqref{eq:pos} yields
\[
 \norm{U_m(1-i,i)}_{L^2(\Omega)}^2+
 \norm{\partial_xU_m(1-i,i)}_{L^2(\Omega)}^2
   =\Im g_m(1-i,i)\to0.
\]
This last identity includes both the cap and the rectangle components.

\section{Harmonic measure and real parameters}\label{s:harm}
The positivity of $ \Im g_m $ (see \eqref{eq:pos}) lets us convert $c \mapsto g_m( c, i ) $ to a bounded holomorphic function using 
the map $g\mapsto(g-i)/(g+i)$
from the upper half-plane to the unit disk. Thus for
$ \Re c > 0 $, $ \Im c <  0 $ we 
put 
\begin{equation}\label{eq:h}
\begin{gathered}
 h_m ( c ) :=1+\re\frac{g_m ( c , i ) -i}{g_m ( c, i ) +i}
   =\frac{2(|g_m ( c, i)|^2+\im g_{m}(c,i) )}{|g_m ( c, i ) +i|^2} , \\
0 \leq h_m ( c ) \leq 2  , \ \ \ \  h_m ( 1 - i ) \to 0 , \ m \to 
   \infty ,
   \end{gathered} 
\end{equation}
where the last statement follows from \eqref{eq:dec}.

The Poisson formula for the lower half-plane $ \Im w <  0 $
evaluated at $ w = - 2i $ is given by 
\[
 F(-2i)=\int_{\R}\frac{2}{\pi(\xi^2+4)}F^*(\xi)d\xi
\]
where $ F $ is a bounded harmonic function on $\im w<0$, and where $F^*$ is its
almost-everywhere defined boundary value -- see for instance
\cite[Chapter~1]{HK}.
The map $w=c^2$ sends $\re c>0$, $\im c<0$ onto $\im w<0$ and
sends $1-i$ to $-2i$. Hence, the change of variables gives
\[
 h_m(1-i)=\int_0^\infty P(s)h_m(s)d s
          +\int_0^\infty P(s)h_m^*(-is)d s.
\]
where 
\begin{equation}\label{eq:P}
 P(s)=\frac{2}{\pi(s^4+4)}\,2s
     =\frac{4s}{\pi(s^4+4)},\qquad s>0.
\end{equation}
We should stress that 
the second term uses boundary limits, not solutions of \eqref{eq:weak} at $\re c=0$. 

Since the boundary values, $  h_m^* ( s ) = h_m ( s )  $
($ g_m $ is defined for $ \Re c >0$, $ \Im c = 0 $ -- see
Lemma \ref{l:pos}) and $ h_m^* ( -i s ) $, $ s > 0 $, are
nonnegative, dropping the second integral gives
\begin{equation}\label{eq:avg}
 \int_0^\infty P(s) h_m ( s) d  s
 \leq  h_m(1-i) \to 0.
\end{equation}
For a compact interval $I\Subset(0,\infty)$, the continuous function
$P$ has a positive minimum on $I$. Hence
\begin{equation}\label{eq:loc}
 \int_I h_m(c) d c \to 0.
\end{equation}

We upgrade this estimate slightly with a modified 
spectral parameter:
\begin{lemma}\label{l:avg}
For every compact interval $I\Subset(0,\infty)$ there is a constant
$C_I$, independent of $m$, such that for, 
\begin{equation}
\label{eq:defHg} \mathcal H ( g ) := 1 + \re \frac{ g - i }{g+i }, \ \ \ \  \nu > 0 ,  \end{equation}
\begin{equation}\label{eq:eng}
 \begin{split}
 \int_I \mathcal H (g_m(c, c\nu + i ) ) d c
 &\leq C_I\int_I \mathcal H( g_m (c, i )) d c
 \leq C_I \mathcal H ( g_m (1-i,i))  \to0.
 \end{split}
\end{equation}
\end{lemma}
\begin{proof}
For each real $c>0$, the function
$z\mapsto\mathcal H(g_m(c,z))$ is nonnegative and harmonic on $\C_+$.
The points $i$ and $c\nu+i$, for $c\in I$, lie in a fixed compact
subset of $\C_+$. Harnack's inequality
\cite[Theorem~1.18]{HK}, applied along a finite chain of disks in
$\C_+$, gives
\[
 \mathcal H(g_m(c,c\nu+i))\leq C_I\mathcal H(g_m(c,i)),
       \qquad c\in I.
\]
The constant is independent of $m$ and $c$. Integrating in $c$ proves
the first inequality. The second one was already established in 
\eqref{eq:avg} using $\min_I P>0$.
\end{proof}

\section{Existence of strong quasimodes}\label{s:res}

\subsection{Correction of the interface defect}
For $c>0$ we put $E_m(c)=(m\pi)^2+c\nu$, and define
the following candidate for a quasimode (see Remark 2 after
Lemma \ref{l:pos}): 
\begin{equation}\label{eq:corr}
 v_m=q_m+c\gamma_tU_m(c,c\nu+i), \ \ \
 (H_c-E_m(c))v_m =i(v_m-q_m),
\end{equation}
where the last identity follows from \eqref{eq:defect} and
\eqref{eq:source}.
Since $v_m\in H^1_0(\Omega)$ and $(H_c-E_m(c))v_m=i(v_m-q_m)$ belongs to
 $L^2(\Omega)$, $v_m\in D(H_c) =H^2(\Omega)\cap H^1_0(\Omega)$. 

The quality of the quasimode is measured by 
\begin{equation}\label{eq:A}
 A_m(c)=\norm{v_m-q_m}^2
       =c^2\gamma_t^2\Im g_m(c,c\nu+i),
\end{equation}
where the last equality comes from \eqref{eq:pos}. We also note that \eqref{eq:corr} gives (using the
resolvent identity $ 
(H-\bar z)^{-1}-(H-z)^{-1}=-2i\,\Im\,z\,(H-\bar z)^{-1}(H-z)^{-1}
$)
\begin{equation}\label{eq:Aid}
 A_m(c)=1-\Im\ip{(H_c-E_m(c)-i)^{-1}q_m}{q_m}, 
\end{equation}
and that
\begin{equation}\label{eq:sp}
 \begin{gathered}
 \norm{v_m}^2=1-A_m(c),\qquad
 \Re\ip{v_m}{q_m}=1-A_m(c).
 \end{gathered}
\end{equation}
In particular, $0\leq A_m(c)<1$, since \eqref{eq:corr} rules out $v_m=0$. 

\subsection{The averaged estimate}
We now fix  $I\Subset(0,\infty)$, and consider
write $g=g_m(c,c\nu+i)$, $ c \in I $. 
For $\Im g\geq0$, the definition \eqref{eq:defHg} gives
\[
 \mathcal H(g)\geq\frac{2\,\Im g}{1+\Im g}
   =\frac{2A_m(c)}{c^2\gamma_t^2+A_m(c)}.
\]
Since $A_m(c)<1$, it follows that
\begin{equation}\label{eq:Af}
 A_m(c)\leq\tfrac12(c^2\gamma_t^2+1)\mathcal H(g)
        \leq C_I\mathcal H(g),\qquad c\in I.
\end{equation}
Integrating \eqref{eq:Af} and applying Lemma~\ref{l:avg} gives
\begin{equation}\label{eq:avA}
 \begin{split}
 \int_I A_m(c)d c
 &\leq C_I\int_I\mathcal H(g_m(c,c\nu+i))d c\\
 &\leq C_I\mathcal H(g_m(1-i,i))\to 0.
 \end{split}
\end{equation}
Thus the intermediate integral is precisely the one in
\eqref{eq:eng}.

\subsection{Normalization and selection of a subsequence}
Writing $A_m=A_m(c)$ and $u_m=v_m/\sqrt{1-A_m}$, we obtain from
\eqref{eq:sp}
\begin{equation}\label{eq:fin}
 \begin{gathered}
 \norm{(H_c-E_m(c))u_m}=\sqrt{\frac{A_m}{1-A_m}},\qquad
 \norm{u_m}_{L^2(C)}^2\leq\frac{A_m}{1-A_m},\\
 \norm{u_m-q_m}^2=2-2\sqrt{1-A_m}.
 \end{gathered}
\end{equation}
The cap estimate uses $q_m|_C=0$ and
$\norm{v_m-q_m}^2=A_m$. The last equality follows from
$\Re\ip{v_m}{q_m}=1-A_m$ and $\norm{u_m}=\norm{q_m}=1$.
Also $u_m\in D(H_c)$.

To choose a common subsequence, use \eqref{eq:avA} to select
$m_k>m_{k-1}$, for $k\geq2$, such that
\[
 \int_{1/k}^{k} A_{m_k}(c)d c\leq2^{-k}.
\]
For each fixed $j\geq2$, Tonelli's theorem gives
$\sum_{k\geq j}A_{m_k}(c)<\infty$ for almost every
$c\in[1/j,j]$. Hence $A_{m_k}(c)\to0$ for almost every $c>0$.
Equations \eqref{eq:fin} prove Theorem~\ref{t:main}.
The change of parameter $a=c^{-1/2}$ preserves null sets on compact
positive intervals, giving the geometric formulation.

\section{Exact eigenfunctions}\label{s:ef}
Now we begin to work toward a more precise version of Theorem~\ref{t:1}. We will show in Theorem~\ref{t:ef} that there exists a sequence of eigenfunctions that localize to the bouncing-ball modes in the closed rectangle. To do so, we first construct a sequence of high-energy eigenfunctions whose $L^2$ mass of the semiclassical $x$-derivative tend to zero (see Lemma~\ref{lem:dx_control}). This means that the sequence must concentrate on the vertical modes.

The strategy is a variation of Hassell's argument in
\cite[\S\S 2--4]{Ha}. Using the averaged estimate \eqref{eq:avA}, we find that for almost every $c > 0$, there exists nonempty spectral windows centered at 
\[E_m(c):=(m\pi)^2+c\nu\]
in which we have control over over the eigenvalue derivative $\lambda'_j(c)$. This in turn controls the $x$ derivative of the corresponding eigenfunction. 

\subsection{Eigenvalue derivatives and counting}
Let $\lambda_j(c)$ denote the eigenvalues of $H_c$, repeated with multiplicity
and arranged in increasing order. 
One can view $H_c$ as a family of Laplacians with respect to an analytic family of metrics on $\Omega$ 
and find that off crossing parameters, $\lambda_j(c)$ can be followed analytically -- compare \cite[\S 2]{Ha}. For a normalized eigenfunction belonging to an
analytic branch, differentiation of the form gives
\begin{equation}\label{eq:hf}
 \lambda_j'(c)=\norm{\partial_x\varphi_j(c)}^2\ge0.
\end{equation}
At a crossing one diagonalizes the form $\ip{\partial_xu}{\partial_xv}$ in the
eigenspace. Isolated crossings form a null set. If a multiple eigenvalue persists
on an interval, this form restricted to the eigenspace is $\lambda_j'(c)$ times the identity.
Consequently the estimates below apply to all eigenfunctions in a given
eigenspace, outside a null set. No simplicity hypothesis is used.

Fix a compact parameter interval $I=[c_-,c_+]\Subset(0,\infty)$. Write
\[
 N_m(c,\eta)=\#\{j:|\lambda_j(c)-E_m(c)|\le\eta\},\qquad 0<\eta\le1.
\]
The Weyl law reads
\[
 N_c(L)=\frac{|\Omega|}{4\pi\sqrt c}L+o(L), \quad N_c(L) := \#\{j : \lambda_j(c) \le L\},
\]
where the $o(L)$ remainder is uniform over $c \in I$. Therefore, it follows that 
\begin{equation}\label{eq:cnt}
 r_m:=\frac{1}{(m\pi)^2}\int_I N_m(c,1)d c\to 0.
\end{equation}
We have the following lemma which controls the average number eigenvalues in a size $\eta$ window of $E_m(c)$ with poorly behaved derivative in $c$. 

\begin{lemma}\label{l:fast}
For $0<\eta\le1$ and $\beta>0$, let
\[
 F_m(c;\eta,\beta)=\#\{j:|\lambda_j(c)-E_m(c)|<\eta,
                       \ \lambda_j'(c)-\nu>\beta (m\pi)^2\}.
\]
There is $C_I$ independent of $m,\eta,\beta$ such that
\begin{equation}\label{eq:fast}
 \int_I F_m(c;\eta,\beta)d c
 \le \frac{C_I\eta+\nu r_m}{\beta}.
\end{equation}
\end{lemma}
\begin{proof}
Put $f_j(c) := \lambda_j(c) - E_m(c) = \lambda_j(c)-(m\pi)^2-c\nu$. Note that $f_j'\ge-\nu$ where
 the derivative exists. Let
\[
 T_\eta(s):= -\eta \mathbf 1_{ (-\infty ,  -\eta)  } + s \mathbf 1 _{[-\eta, \eta ] } + \eta \mathbf 1_{(\eta, \infty ) } . 
\]
Then 
$ \dfrac{d}{dc} T_\eta(f_j(c))  = \boldsymbol1_{\{|f_j| < \eta\}} f'_j(c)$, so
\[
 \int_I \boldsymbol1_{\{|f_j|<\eta\}}f_j'd c
   =T_\eta(f_j(c_+))-T_\eta(f_j(c_-))\le2\eta.
\]
Using $f_j'\ge-\nu$, we therefore obtain
\[
 \int_I\boldsymbol1_{\{|f_j|<\eta\}}(f_j')_+d c
 \le 2\eta+\nu\int_I\boldsymbol1_{\{|f_j|<\eta\}}d c.
\]
Chebyshev's inequality then gives us control over the size of the set of parameters for which $f'_j = \lambda'_j(c) - \nu$ is poorly controlled:
\[
|\{c\in I: |f_j(c)|<\eta,\ f_j'(c)>\beta(m\pi)^2\}|\\
\leq\frac{1}{\beta(m\pi)^2}\left(2\eta+\nu\int_I\boldsymbol1_{\{|f_j|<\eta\}}dc\right).
\]
By the Weyl law, there exists $C_I > 0$ depending only on $I$ such that if $|\lambda_j(c) - E_m(c)| < \eta$ for some $c \in I$, then $j < C_I m^2$. Therefore, we can bound
\begin{equation}
\begin{aligned}
\int_I F_m(c;\eta,\beta)\,dc 
&\leq\sum_{j<C_I m^2}|\{c\in I: |f_j(c)|<\eta,\ f_j'(c)>\beta(m\pi)^2\}|\\
&\leq\frac{C_I\eta+\nu r_m}{\beta}.
\end{aligned}
\end{equation}
Here, $C_I > 0$ may change from line to line but remains independent of $m, \eta, \beta$. This recovers~\eqref{eq:fast}.
\end{proof}

\subsection{Selection in shrinking windows}
From~\eqref{eq:hf}, we see that controlling $\partial_x \varphi_j(c)$ amounts to controlling $\lambda'_j(c)$. Lemma~\ref{l:fast} tells us that on average in the parameter $c$, we have good control in a spectral window as long as the size of the window $\eta$ is small. In the following lemma, we will pick a good sequence of shrinking spectral windows so that we have the desired control with probability $1$. 
\begin{lemma}\label{lem:dx_control}
For almost every $c>0$ there are integers $m_k\to\infty$ and normalized
real eigenfunctions $\varphi_k\in H^2 ( \Omega ) \cap 
H_0^1 ( \Omega ) $ such that
\begin{equation}\label{eq:ef}
 \begin{gathered}
 H_c\varphi_k=\lambda_k\varphi_k,\qquad
 \lambda_k=(m_k\pi)^2+c\nu+o(1),\\
 \lambda_k^{-1}\norm{\partial_x\varphi_k}_{L^2(\Omega)}^2\to0.
 \end{gathered}
\end{equation}
On each compact parameter interval the integers $m_k$ may be chosen
independently of $c$, outside a null set.
\end{lemma}
\begin{proof}
Recall the quantity $A_m(c)$ from~\eqref{eq:Aid}. Set $\alpha_m(I)=\int_I A_m(c)d c$, which tends to zero by
\eqref{eq:avA}.
The spectral theorem gives
\begin{equation}\label{eq:sd}
 A_m(c)=\int_\R\frac{(s-E_m(c))^2}{1+(s-E_m(c))^2}
                     d\mu_{q_m,c}(s),
 \qquad 0\le A_m(c)<1.
\end{equation}
Let $\Pi_m(c,\eta)=\boldsymbol1_{[E_m(c)-\eta,E_m(c)+\eta]}(H_c)$. Then
\begin{equation}\label{eq:prj}
 \norm{(1-\Pi_m(c,\eta))q_m}^2
 \le\frac{1+\eta^2}{\eta^2}A_m(c).
\end{equation}
In particular, it follows from Chebyshev's inequality that 
\begin{equation}\label{eq:empty_probability}
\begin{split}
\epsilon_m := |\{c\in I:\Pi_m(c,\eta)q_m=0\}|
&\leq |\{c\in I:(1+\eta^2)A_m(c)\geq \eta^2 \}|\\
&\leq\frac{1+\eta^2}{\eta^2}\alpha_m(I).
\end{split}
\end{equation}
In other words, the measure of parameters for which the spectral window is empty is at
most $(1+\eta^2)\alpha_m(I)/\eta^2$.

For large $m$, choose
\[
\eta_m=\big(\alpha_m(I)+m^{-1}\big)^{1/4},\qquad
 \beta_m=(\eta_m+r_m)^{1/2}.
\]
Both tend to zero. Moreover, choosing $\epsilon_m$ with $\eta_m$ in the spectral projector $\Pi_m(c, \eta_m)$, we from~\eqref{eq:empty_probability} that $\epsilon_m \to 0$. By Lemma~\ref{l:fast}, the measure of parameters with any eigenvalue
in the window satisfying $\lambda_j'-\nu>\beta_m(m\pi)^2$ also tends to zero, that is
\[ \rho_m:=|\{c\in I:F_m(c;\eta_m,\beta_m)\geq1\}|\to0.\]
We can choose a sequence $m_k\to\infty$ so that
$
\sum_k\left(\epsilon_{m_k} +\rho_{m_k}
\right)<\infty.
$
Then by the Borel--Cantelli lemma, we see that for almost every $c\in I$,
all sufficiently large $k$ then have a nonempty window and every eigenvalue
in its interior satisfies
\begin{equation}\label{eq:slp}
 \lambda_j'(c)\le\nu+\beta_{m_k}(m_k\pi)^2.
\end{equation}
An endpoint equality causes no difficulty: unless the analytic function $\lambda_j(c)-E_m(c)$ is constant,
it occurs on a null set; if it persists then $\lambda_j'=\nu$.

Choose any normalized real eigenfunction in the window. Equations
\eqref{eq:hf} and \eqref{eq:slp} give
\begin{equation}
 H_c\varphi_k=\lambda_k\varphi_k,\qquad
 \lambda_k=(m_k\pi)^2+c\nu+o(1),\qquad
 \lambda_k^{-1}\norm{\partial_x\varphi_k}^2\to0
\end{equation}
as desired.
\end{proof}
Observe that \eqref{eq:prj} also shows that the projections of $q_{m_k}$
onto these windows approach $q_{m_k}$ in norm for almost every $c$.
This does \emph{not} assert that a single eigenfunction approaches $q_{m_k}$.
Finally take a countable cover of $(0,\infty)$ by compact intervals.

\section{Localization of the eigenfunctions}\label{s:efloc}
We work on the full physical domain $S_a$ to avoid the artificial
corners of a quarter-domain. Let $\psi_k$ be the normalized odd
reflections of the quarter-domain eigenfunctions $\varphi_k$ satisfying~\eqref{eq:ef}. Recall that such eigenfunctions were selected in Section~\ref{s:ef}.
Set $h_k=\lambda_k^{-1/2}$. In particular, put
\[\psi_k(X,Y)=(2\sqrt a)^{-1}\varphi_k(X/a-t,Y), \quad X, Y \ge 0\]
in the first quadrant and extend oddly across
both axes. The change of variables gives
\begin{equation}\label{eq:phx}
 \norm{h_k\partial_X\psi_k}_{L^2(S_a)}^2
 =c\,\frac{\norm{\partial_x\varphi_k}_{L^2(\Omega)}^2}{\lambda_k}
 \to0.
\end{equation}
We work in the full stadium coordinates from now on, and by a slight abuse of notation, we write $ ( x, y ) $ to mean the full stadium $ ( X, Y ) $-coordinates from now on.

Extend $\psi_k$ by zero to $\R^2$. The extensions belong to
$H^1(\R^2)$ and satisfy $\norm{\psi_k}_{L^2} =1$ and
$\norm{h_k\nabla\psi_k}_{L^2}=1$.
The Dirichlet semiclassical measure construction for $C^{1,1}$
boundaries \cite{GL} gives, after passage to a subsequence, a
probability measure $\mu$ with support on the unit cosphere bundle of $\overline{S_a}$. Moreover, by~\eqref{eq:phx}, we further see that
\begin{equation}\label{eq:supp_mu}
 \supp \mu \subset \{(x,y;\xi):(x,y)\in\overline{S_a},\ \xi = (0, \pm 1)\}.
\end{equation}

Following \cite[Th\'eor\`eme~15(iii) and (iv$'$)]{Bu97}, we see that 
$\mu$ gives no mass to hyperbolic boundary points and its support
propagates along generalized billiard trajectories at smooth boundary
points. A vertical ray in a wing reflects into a nonvertical direction,
contradicting \eqref{eq:supp_mu}. On the elliptical arcs away from the joins, vertical
directions are hyperbolic except at the two outer tips. At a tip the
ray glides along the ellipse and immediately becomes nonvertical,
again contradicting \eqref{eq:supp_mu}. Thus
\begin{equation}\label{eq:closedwings}
 \mu\bigl(\{(x,y;\xi):(x,y)\in\overline{S_a},\ |x|>at\}\bigr)=0,
\end{equation}
including the curved boundary. Only the tip argument uses gliding
propagation. The remaining horizontal boundary, including the four joins, is handled by the integration-by-parts argument in the proof of Theorem \ref{t:ef}  below.

Combining this with Lemma~\ref{lem:dx_control}, we have the following
strengthening of Theorem~\ref{t:1}:
\begin{theorem}\label{t:ef}
There are integers $m_k\to\infty$ such that for almost every $c>0$ there exist normalized
real eigenfunctions $\psi_k\in H^2 (S_a) \cap 
H_0^1 (S_a) $ such that
\begin{equation*}
-\Delta \psi_k=\lambda_k\psi_k,\qquad
 \lambda_k=(m_k\pi)^2+c\nu+o(1).
\end{equation*}
In addition, every semiclassical measure of this sequence has total
mass one, gives zero mass to $\partial S_a$, and is supported on
\[
 \overline{\mathcal B}_a=
 \{(x,y;0,\pm1): |x|\leq at,\ |y|\leq1\}.
\]
In particular, \eqref{eq:wings} holds. 
\end{theorem}
\begin{proof}
By \eqref{eq:closedwings} and interior propagation, the spatial marginal
$\rho$ satisfies
\[
 d\rho\big|_{\{|y|<1\}}=d\sigma(x)\,dy,\qquad
 \supp\sigma\subset[-at,at],
\]
including possible mass on $x=\pm at$. It remains to exclude mass on
$y=\pm1$, including the joins.

So, we write $u=\psi_k$, $h=h_k$, let $n$ be the outward
unit normal, and define $b(y)=y(1-y^2)$.
Using that $(-h^2\Delta-1)u=0$ and the boundary condition $u|_{\partial S_{a}} =0$, it follows from Green's formula that 
\[
\begin{split}
 \int_{\partial S_a}bn_y|h\partial_nu|^2\,ds
 &={\rm Re}\int_{S_a}[-h^2\Delta,b\partial_y]u\,\overline u\,dxdy\\
 &=2\int_{S_a}b'|h \partial_y u|^2\,dxdy
   -\frac{h^2}{2}\int_{S_a}b'''|u|^2\,dxdy.
\end{split}
\]
Here we used
$[-h^2\Delta,b\partial_y]=-2h^2b'\partial_y^2-h^2b''\partial_y$
and integrated by parts. Testing the eigenfunction equation against
$b'\overline u$ also gives
\[
 \int_{S_a}b'\bigl(|h\nabla u|^2-|u|^2\bigr)\,dxdy
 =\frac{h^2}{2}\int_{S_a}b'''|u|^2\,dxdy.
\]
Combining these identities and using $\norm u=1$, we obtain
\begin{equation}\label{eq:b_lim}
 2\int_{S_a}b'|u|^2\,dxdy
 =\int_{\partial S_a}bn_y|h\partial_nu|^2\,ds
   +O(\norm{h \partial_x u}^2+h^2).
\end{equation}

On $\partial S_a$ we have $0\leq bn_y\leq xn_x$. The same Green
identity with $x\partial_x$, for which
$[-h^2\Delta,x\partial_x]=-2h^2\partial_x^2$, gives Rellich's identity:
\[
 \int_{\partial S_a}xn_x|h\partial_nu|^2\,ds
 =-2h^2{\rm Re}\int_{S_a} \partial_x^2 u\overline u\,dxdy
 =2\norm{h \partial_x u}^2\to0.
\]
Thus \eqref{eq:b_lim} gives $\int b'(y)\,d\rho=0$.
The interior contributes zero since
$\int_{-1}^1b'(y)\,dy=0$, and the curved boundary has zero mass by
\eqref{eq:closedwings}. Since $b'(\pm1)=-2$,
\[
 0=-2\rho\bigl(\{(x,y):|x|\leq at,\ y=\pm1\}\bigr),
\]
which excludes the horizontal boundary and all four joins.
\end{proof}

Note that for each fixed $\varepsilon>0$, vanishing mass on
$\{|x|\geq at+\varepsilon\}$ follows by a continuous cutoff and compactness, so Theorem~\ref{t:1} indeed follows from Theorem~\ref{t:ef}.
Again, we stress that this does not imply that
$\norm{\varphi_k-q_{m_k}}\to0$ for individual eigenfunctions.

\end{document}